\documentclass[12pt,a4paper,oneside]{amsart}
\usepackage[utf8]{inputenc}

\usepackage{amsmath,amscd,hyperref}
\usepackage{amsfonts}
\usepackage{amssymb}
\usepackage{amsthm}
\usepackage{cite}
\usepackage[margin=1in]{geometry}
\usepackage{mathtools}
\DeclarePairedDelimiter{\abs}{\lvert}{\rvert}
\usepackage{footnote}
\usepackage{colonequals}

\DeclareMathOperator{\real}{\mathrm{Re}}
\DeclareMathOperator{\im}{\mathrm{Im}}
\numberwithin{equation}{section}

\theoremstyle{plain} \newtheorem{thm}{Theorem}[section]
\newtheorem{lemma}[thm]{Lemma}
\newtheorem{prop}[thm]{Proposition}
\newtheorem{corollary}[thm]{Corollary}
\newtheorem{fact}[thm]{Fact}

\theoremstyle{definition}\newtheorem{ex}[thm]{Example}
\newtheorem{remark}[thm]{Remark}

\newtheorem{definition}[thm]{Definition}
\newtheorem{setup}[thm]{Set-Up}
\newtheorem{construction}[thm]{Construction}
\newtheorem*{convention*}{Convention}
\newtheorem*{acknowledgements*}{Acknowledgements}

\DeclareMathOperator{\Res}{Res}
\DeclareMathOperator{\rank}{rk}
\DeclareMathOperator{\gr}{gr}

\newcommand{\can}{\omega_{Y/X}}
\newcommand{\D}{\Delta}
\newcommand{\cM}{\mathcal{M}}
\newcommand{\cV}{\mathcal{V}}

\newcounter{tmp}

\begin{document}

\title{Hodge modules and singular hermitian metrics}

\author{Christian Schnell and Ruijie Yang}

\maketitle

\section*{Introduction}

Motivated by the Iitaka conjecture, there has been a lot of interests in studying the positivity of direct images of relative pluri-canonical bundles using Hodge theory. Among results about algebraic positivity, it is known that they are weakly positive \cite{Viehweg}. Recently, the authors in \cite{CP,HPS} emphasize a different aspect of positivity from the metric point of view. Let $f:Y\to X$ be a projective and surjective holomorphic map between complex manifolds. Given a holomorphic line bundle $(L,h)$ on $Y$, where $h$ is a singular hermitian metric with semi-positive curvature, the authors in \cite{HPS} construct a singular hermitian metric on $f_{\ast}(\omega_{Y/X}\otimes L \otimes \mathcal{I}(h))$. Furthermore, they show that this new metric satisfies the ``minimal extension property'' (see Section~\ref{section:metric positivity}). In particular, it has semi-positive curvature, which generalizes the work of \cite{PT}. This metric positivity then implies Viehweg's results on weak positivity by \cite[Theorem 2.5.2]{PT}.

Let $k$ be the dimension of a general fiber of $f$ and let $(L,h)$ be $\mathcal{O}_Y$ equipped with the trivial metric. The construction in \cite{HPS} gives a singular hermitian metric on $f_{\ast}(\can)$, which is the lowest piece in the Hodge filtration of $R^kf_{\ast}(\underline{\mathbb{C}})$. This singular hermitian metric on $f_{\ast}(\can)$ actually comes from the Hodge metric on the smooth part of the fibration. It is natural to ask whether or not the Hodge metric on the lowest pieces in the Hodge filtration of \textit{any} polarized variation of rational Hodge structures satisfies the ``minimal extension property''. In this paper, we would like to give an affirmative answer using the language of Hodge modules.

Let $X$ be a complex manifold and let $\cM$ be a polarized pure Hodge module on $X$ with strict support $X$. Let $p$ be the smallest integer such that $F_p\cM \neq 0$. Assume that $\cM$ is smooth (i.e. the perverse sheaf underlying $\cM$ is a local system up to a shift) outside a closed analytic subset $Z\subsetneq X$ and let $j:X\setminus Z \hookrightarrow X$ be the open embedding. The polarization on $\cM|_{X\setminus Z}$ provides the vector bundle $F_p\cM|_{X\setminus Z}$ with a smooth hermitian metric $h$ (called Hodge metric). Our main result is
\begingroup
\setcounter{tmp}{\value{thm}}
\setcounter{thm}{0} 
\renewcommand\thethm{\Alph{thm}}
\begin{thm}\label{thm:main}The canonical morphism of $\mathcal{O}_X$-modules
\[ F_p\cM \to j_{\ast}(F_p\cM|_{X\setminus Z})\]
induces an isomorphism between $F_p\cM$ and the subsheaf of $j_{\ast}(F_p\cM|_{X\setminus Z})$ consisting of sections of $F_p\cM|_{X\setminus Z}$ which are locally $L^2$ near $Z$ with respect to $h$ and the standard Lebesgue measure on $X$.
\end{thm}
\endgroup

One can also formulate this in the classical language of variations of Hodge structures. Let $Z\subsetneq X$ be a closed analytic subset of $X$ and let $j:X\setminus Z \hookrightarrow X$ be the open embedding. Let $\cV$ be a polarized variation of rational Hodge structures on $X\setminus Z$. Saito shows that $\cV$ uniquely extends to a polarized pure Hodge module $\cM$ on $X$ with strict support $X$.  Let $E$ be the holomorphic vector bundle corresponding to the lowest nonzero piece in the Hodge filtration of $\cV$ and let $(\cM,F_{\bullet}\cM)$ be the filtered $D_{X}$-module underlying $\cM$. Saito also shows that $E=F_p\cM|_{X\setminus Z}$, where $F_{p}\cM$ is a torsion-free sheaf on $X$ (see Section~\ref{section:HM}). The canonical morphism of $\mathcal{O}_X$-modules in Theorem \ref{thm:main} corresponds to
\[ F_p\cM  \to j_\ast E. \]
Since the Hodge metric $h$ is known to be Nakano semi-positive, as a consequence of the sharp Ohsawa-Takegoshi Theorem \cite{Blocki, GZ}, we have
\begingroup
\setcounter{tmp}{\value{thm}}
\setcounter{thm}{1} 
\renewcommand\thethm{\Alph{thm}}
\begin{corollary} \label{cor:minimal extension property}With the notation above, the Hodge metric $h$ extends as a singular Hermitian metric on $F_p\cM$ which satisfies the ``minimal extension property" and has semi-positive curvature.
\end{corollary}
\endgroup
If one assumes $D$ is simple normal crossing, the curvature property of the Hodge metric for a $\mathbb{C}$-VHS is proven by Brunebarbe \cite[Theorem 1.4]{Br}.  

The notion of ``minimal extension property'' arises from the work of Ohsawa-Takegoshi with sharp estimates by \cite{Blocki, GZ} and it is a slight strengthening of Griffiths semi-positivity for singular hermitian metrics (i.e. with semi-positive curvature). The key point of this notion is about the ability to extend sections over the singular locus of torsion-free sheaves while controlling $L^2$ norms in a precise way. This notion plays an important role in the proof of the Iitaka conjecture for algebraic fiber spaces over abelian varieties \cite{CP,HPS}. One reason to investigate the ``minimal extension property'' of Hodge metrics is that we hope one can apply it to other related situations. On the other hand, it is a general philosophy that the theory of Hodge modules is closely related to $L^2$ methods. For one thing, by Saito's inductive construction, cohomology of Hodge modules on arbitrary manifolds can be reduced to Hodge modules on curves which are studied by Zucker \cite{Zucker} using $L^2$ cohomology. On the other hand, there are lots of work around close relations between $D$-modules and multiplier ideal sheaves (see \cite{BS2003,MM16} for example).

Let us briefly sketch the proof. We will use the formulation after Theorem \ref{thm:main}. Since it is a local statement and any VHS away from codimension 2 automatically extends, we can assume $X$ is a unit ball in $\mathbb{C}^n$ and $Z$ is a divisor on $X$. To prove Theorem \ref{thm:main}, first we use the direct image theorem of Hodge modules \cite{Saito88} to reduce to the case where $Z$ is simple normal crossing. After a finite base change and some calculations, we further reduce to the case with unipotent monodromy. To conclude the proof, we use results by Schmid \cite{Schmid} on asymptotics of Hodge metrics to analyze the coefficient functions of $L^2$ sections.



In $\S \ref{sec:prelim}$ we will review some background. In $\S \ref{section:proof}$ we give the proof of Theorem~\ref{thm:main} and Corollary~\ref{cor:minimal extension property}.

\begin{convention*}
We will use left $D$-modules throughout the paper, as they are more natural from the metric point of view.
\end{convention*}

\begin{acknowledgements*}
We would like to thank Nathan Chen, Robert Lazarsfeld and Lei Wu for reading a draft of the paper. We would like to thank Junchao Shentu for pointing out an error in the previous version of the paper.
\end{acknowledgements*}

\section{Preliminaries}\label{sec:prelim}

In this section, we will set up notation and recall some background.

\subsection{Variation of Hodge Structures}\label{sec:definition of VHS}
Let $X$ be a complex manifold. A \textit{polarized variation of rational Hodge structures} (VHS) on $X$ consists of the following data:

\begin{enumerate}

\item A local system $V_{\mathbb{Q}}$ of finite dimensional $\mathbb{Q}$-vector spaces; 
\item A holomorphic vector bundle $\cV$ with a flat connection $\nabla: \cV \to \Omega^1_X \otimes \cV$;

\item A finite decreasing filtration $F^{\bullet}\cV$ by holomorphic subbundles satisfying the Griffiths transversality
\[ \nabla F^{p}\cV \subset \Omega^1_X \otimes F^{p-1}\cV;\]

\item A flat non-degenerate bilinear form $S: V_{\mathbb{Q}} \otimes_{\mathbb{Q}} V_{\mathbb{Q}} \to \mathbb{Q}.$ 

\end{enumerate}
They are related in the following way:
\begin{itemize}

\item The local system of $\nabla$-flat holomorphic sections of $\cV$ is isomorphic to $V_{\mathbb{Q}}\otimes_{\mathbb{Q}} \mathbb{C}$. In particular, 
\[ \cV \cong V_{\mathbb{Q}}\otimes_{\mathbb{Q}} \mathcal{O}_X\]
and $S$ extends to $\cV$ in a $C^{\infty}$ way.
\item $\cV$ admits the Hodge decomposition $ \cV=F^p \oplus \overline{F}^{k-p+1}$,
where $\overline{F}$ is the conjugate of $F$ relative to $V_{\mathbb{R}}=V_{\mathbb{Q}}\otimes_{\mathbb{Q}}\mathbb{R}$ and $k$ is the weight of the variation.
 
\item $S^h(\cdot,\cdot) \colonequals i^{-k}S(\cdot, \bar{\cdot})$ is a hermitian form such that the Hodge decomposition is $S^h$-orthogonal and $(-1)^pS^h$ is positive definite on $\cV^{p,k-p}$. In particular, such a polarization determines a smooth hermitian metric $h$ on $\cV$:
\begin{eqnarray}\label{eqn:filtration metric}
 h\colonequals \sum_p (-1)^pS^h|_{\cV^{p,k-p}}.
\end{eqnarray}
We will call $h$ the \textit{Hodge metric} on $\cV$. For $v\in \cV_x$, $\abs{v}_{h,x}$ means the length of this vector.
\item For every $x\in X$, $(\cV_x,V_{\mathbb{Q},x}, F^{\bullet}\cV_x)$ defines a rational Hodge structure of weight $k$ which is polarized by the bilinear form $S_x$.
\end{itemize}

\begin{remark}
In the rest of the paper, when we say $\cV$ is a VHS, we actually mean a polarized variation of rational Hodge structures with data $(V_{\mathbb{Q}},\cV,\nabla,F^{\bullet}\cV,S)$. 
\end{remark}

\subsection{Deligne's canonical lattice}\label{section:can}
In this section, we will assume $X=\D^{n}$ and $U=(\D^{\ast})^n$. Let $s_1,\ldots,s_n$ be the coordinates on $X$. Then $D\colonequals X\setminus U$ is a simple normal crossing divisor so that $D=\cup D_j$, where $D_j=\{ s_j = 0\}$ for $1\leq j \leq n$. 

Let $\cV$ be a holomorphic vector bundle over $U$ with a flat connection $\nabla: \cV \to \Omega^1_U \otimes \cV $. The fundamental group of $U$ is canonically isomorphic to $\mathbb{Z}^n$, which acts on the fiber of $\cV$ by parallel transport. Let $T_j$ be the operator corresponding to the $j$-th standard generator of  $\mathbb{Z}^n$. We say $(\cV,\nabla)$ has \textit{quasi-unipotent monodromy} if each $T_j$ is quasi-unipotent. An interval $I\subseteq \mathbb{R}$ is said to be of \textit{length one} if the natural quotient map $\mathbb{R} \to \mathbb{R}/\mathbb{Z}$ restricts to an isomorphism of sets on $I$. Deligne\cite{Deligne} proved the following fact:

\begin{thm} Let $\cV$ be a holomorphic vector bundle over $U$ with a flat connection $\nabla$ so that $(\cV,\nabla)$ has quasi-unipotent monodromy. For any interval $I$ of length one, there exists a holomorphic bundle $\cV_{I}$ on $X$ (which is unique up to a unique isomorphism) extending $\cV$ such that $\nabla$ extends to
\[ \nabla: \cV_{I} \to \Omega^1_{X}(\log D) \otimes \cV_{I},\]
with only logarithmic poles and the eigenvalues of the residue operator
\[ \Res_{D_j}\nabla \colonequals s_j\nabla_{\frac{\partial}{\partial s_j}}: \cV_{I}|_o \to \cV_{I}|_o \]
are contained in $I$ for each j, where $o$ is the origin of $X$.
\end{thm} 

\begin{definition}\label{definition:canonical lattice} We define \textit{Deligne's canonical lattice} to be the associated holomorphic vector bundle $\cV_{I}$ for any interval $I$ of length one. For any $\beta \in \mathbb{R}$,  we write $\cV^{\beta} \colonequals \cV_{[\beta,\beta+1)}$ and $\cV^{>\beta} \colonequals \cV_{(\beta,\beta+1]}$. 
\end{definition}

For our purposes, the following explicit construction of $\cV^{>-1}$ will be useful. First, we need to define the logarithm of a quasi-unipotent operator.
\begin{construction}\label{def:log} 
Let $V$ be a $\mathbb{C}$-vector space and let $T$ be a quasi-unipotent operator on $V$. Consider the Jordan decomposition 
\[ T=T_s \cdot T_u\]
where $T_s$ is semisimple and $T_u$ is unipotent. We define the \textit{quasi-unipotency index} $m$ to be the smallest positive integer such that $T_s^m=\mathrm{Id}$.

\begin{enumerate}

\item Since $T_s$ is semisimple over $\mathbb{C}$, there is a decomposition
\[ V=\oplus_{\alpha}V_{\alpha}  \quad \mathrm{with} \quad T_{s}|_{V_{\alpha}}=\lambda_\alpha \cdot \mathrm{Id}|_{V_{\alpha}},\] 
where $\lambda_\alpha$ are the eigenvalues of $T_s$. Because $\lambda_{\alpha}^m=1$, we can choose integers $k_\alpha$ satisfying $0 \leq k_\alpha \leq m-1$ such that 
\[ \lambda_\alpha = e^{-2\pi \sqrt{-1}k_\alpha/m}.\]
We define $\log T_s$ by its action on eigenspaces:
\[ \log T_s|_{V_{\alpha}} \colonequals -(2\pi \sqrt{-1}k_\alpha/m) \cdot \mathrm{Id}|_{V_{\alpha}}.\]
\item Since $T_u$ is unipotent, we can define its logarithm as a convergent series
\[ \log T_u= - \sum_{k\geq 1}^{\infty} \frac{1}{k}(\mathrm{Id}-T_u)^k.\]
\end{enumerate}
Then we define the logarithm of $T$ to be  
\[ \log T \colonequals \log T_s + \log T_u,\]
where $\log T_s$ is semisimple and $\log T_u$ is nilpotent.
\end{construction}

Now we begin the construction of $\cV^{>-1}$. Let us consider the universal covering map of $U=(\D^{\ast})^n$:
\[ p: \mathbb{H}^n \to (\D^{\ast})^n, \quad (z_1,\ldots,z_n) \mapsto (e^{2\pi \sqrt{-1}z_1},\ldots, e^{2\pi \sqrt{-1}z_n}).\]
The fundamental group $\mathbb{Z}^n$ acts on $\mathbb{H}^n$ by the rule
\[ (a_1,\ldots,a_n)\cdot (z_1, \ldots, z_n) = (z_1+a_1,\ldots, z_n+a_n).\]
Let $V \colonequals H^0(\mathbb{H}^n,p^\ast\cV)^{p^\ast \nabla}$ to be the space of $p^\ast \nabla$-flat sections of $p^\ast\cV$ on $\mathbb{H}^n$, which trivializes $p^{\ast}\cV$:
\[ p^{\ast}\cV \cong V\otimes \mathcal{O}_{\mathbb{H}^n}.\]
Therefore the sections of $\cV$ over $(\D^{\ast})^n$ correspond to holomorphic maps $\sigma: \mathbb{H}^n \to V$ with the property that $\sigma(z+e_j)=T_j \cdot \sigma(z)$ for all $z\in \mathbb{H}^n$ and all $j=1,\ldots,n$.

We set $R_j \colonequals \log T_j$ by Construction \ref{def:log}. For each $v \in V$, the holomorphic map
\[ \sigma: \mathbb{H}^n \to V, \sigma(z)=e^{\sum z_jR_j}v, \]
has the required property because $\sigma(z+e_j)=e^{R_j}\cdot \sigma(z)=T_j\cdot \sigma(z)$. Then $\cV^{>-1}$ is defined to be the vector bundle over $X=\D^n$ generated by all sections of this form, i.e.
\[ \cV^{>-1} \cong (e^{\sum z_jR_j}V)\otimes_{\mathbb{C}} \mathcal{O}_X.\]
Note that by the construction,
\[ \Res_{D_j}\nabla = \frac{R_j}{2\pi\sqrt{-1}}.\]
Since the eigenvalues of $R_j$ are of the form $-(2\pi \sqrt{-1}k/m)$, where $0\leq k\leq m-1$, the eigenvalues of $\Res_{D_j}\nabla$ all lies in $(-1,0]$. 
Therefore $\cV^{>-1}$ satisfies the conditions in Deligne's theorem.

\begin{remark}\label{remark:unipotent}
If all $T_j$'s are unipotent, we say that $(\cV,\nabla)$ has \textit{unipotent monodromy}. In this case, the eigenvalues of $\Res_{D_j}\nabla$ are all zero. Therefore
\[  \cV^{>-1} = \cV^{0}.\]
\end{remark}

\subsection{Asymptotics of Hodge metrics}\label{section:SL2orbits}
In this section, we will review the results on asymptotics of Hodge metrics.

Let $\cV$ be a $\mathbb{Q}$-VHS over $\D^{\ast}\times \D^{n-1}$ with unipotent monodromy operator $T_1$. Let $N_1\colonequals \log T_1$. Denote the universal covering map to be
\[p: \mathbb{H}\times \D^{n-1} \to \D^{\ast} \times \D^{n-1}.\]
Let $V\colonequals H^0(\mathbb{H}\times \D^{n-1}, p^{\ast}\cV)^{p^{\ast}\nabla}$, where $\nabla$ is the connection associated to $\cV$ and let $W$ be the limit weight filtration on $V$. Let 
\[ \Phi: \mathbb{H}\times \D^{n-1} \to \mathbf{D}\]
be the corresponding period map, where $\mathbf{D}$ is the classifying space of polarized Hodge structures associated to $\cV$. Each point $F\in \mathbf{D}$ determines a hermitian inner product $H_F$ on $V$ as in (\ref{eqn:filtration metric}). Therefore one can represent the pullback of the Hodge metric $p^\ast h$ on $\cV$ by $H_{\Phi(z)}$, where $z$ is the coordinates on $\mathbb{H}\times \Delta^{n-1}$. Schmid \cite[Theorem 4.9]{Schmid} proved that

\begin{thm}[Schmid]\label{thm:one dimension hodge metric}

There is a hermitian inner product $Q$ on $V$ such that over any region of the form 
\[ \mathbb{H}_\epsilon \times \D^{n-1} \]
where $\mathbb{H}_\epsilon \colonequals \{ z_1 \in \mathbb{H} : \im z_1 >  \epsilon \}$, $\abs{v}_{H_{\Phi(z)}}^2$ is mutually bounded with
\[ \sum_{l \in \mathbb{Z}} \left(-\log \abs{s_1}\right)^{l}\abs{v}^2_{Q_l}.\]
where $v\in V$, $s_1=e^{2\pi \sqrt{-1} z_1}$.  $Q=\oplus Q_l$ corresponds to the grading induced by the non-canonical isomorphism $V\cong  \oplus_{l\in \mathbb{Z}} \mathrm{Gr}^{W}_{l}V$.
\end{thm}

\begin{remark}
Here we present the Theorem as in \cite[Theorem 5.1]{CKSurvey}. In \cite[Theorem 4.9]{Schmid}, Schmid only proved the case $n=1$. But it is known that all the constants in Schmid's arguments are uniform in any family of VHS over the punctured disk with compact parameter space. The reader can refer to \cite[Lemma 8.19]{Schmid} for the nilpotent orbit theorem. See also \cite[Page 465, Page 482]{CKS}.
\end{remark}

\begin{lemma}\label{lemma:refined version of Schmids Theorem}
With the notations in Theorem \ref{thm:one dimension hodge metric}. Let $F^q\cV^{>-1}$ be the smallest nonzero piece in the Hodge filtration of the canonical extension $\cV^{>-1}$ . Then $F^q\cV^{>-1}$ has a trivialization by holomorphic sections $\sigma_{\alpha}$ such that if one takes any holomorphic section $\sigma$ of $F^q\cV^{>-1}$ over $\Delta^{\ast}\times \Delta^{n-1}$ and write it as 
\[ \sigma=\sum f_{\alpha} \sigma_{\alpha}\]
with $f_{\alpha}$ being holomorphic sections, the pointwise Hodge norm $|\sigma|^2_{h}$ is mutually bounded with
\[ \sum_{\alpha} |f_{\alpha}|^2(-\log |s_1|)^{n_{\alpha}}\]
where all $n_{\alpha}$ are non-negative integers.
\end{lemma}

\begin{proof} We first would like to describe a way to represent the period map $\Phi$  in terms of the limit mixed Hodge structure $W$ (see \cite[\S 3.2]{Schnell12} and \cite[Page 71]{CKSurvey}). Recall that $z_1$ and $s_1$ are coordinates on $\mathbb{H}$ and $\Delta^{\ast}$, and $z$ is the coordinates on $\mathbb{H}\times \Delta^{n-1}$. By the Nilpotent Orbit Theorem \cite[Theorem 4.12]{Schmid}, we have
\[ e^{-z_1N_1}\Phi(z)=\Psi(s), \]
with $\Psi:\Delta^n \to \check{D}$ holomorphic. Then $(W,\Psi(0))$ is a $\mathbb{Q}$-mixed Hodge structure and we denote
\[ F_{\mathrm{lim}}\colonequals \Psi(0) \in \check{D}.\]
The subspaces $I^{p,q}$ associated to the mixed Hodge structure $(W,F_{\mathrm{lim}})$, defined by
\[I^{p,q}=F_{\mathrm{lim}}^p\cap W_{p+q}\cap (\overline{F}^q_{\mathrm{lim}}\cap W^{p+q}+\sum_{j\geq 1} \overline{F}^{q-j}_{\mathrm{lim}}\cap W_{p+q-j-1})\]
give a bigrading of $(W,F_{\mathrm{lim}})$ in the sense that
\begin{equation}\label{eqn:decomposition of W and F}
W_{\ell} = \bigoplus_{a+b\leq \ell}I^{a,b}, \quad F^p_{\mathrm{lim}}=\bigoplus_{a\geq p}I^{a,b}.
\end{equation}
This induces a decomposition on the complexified Lie algebra $\mathfrak{g}$ of $\mathrm{Aut}(V_{\mathbb{R}},S)$, where $S$ is the bilinear pairing on $V_{\mathbb{R}}$ (see \S \ref{sec:definition of VHS}),
\[ \mathfrak{g}=\bigoplus_{p,q} \mathfrak{g}^{p,q},\]
with $\mathfrak{g}^{p,q}$ consisting of those $X$ that satisfy $XI^{a,b}\subseteq I^{a+p,b+q}$. Define
\[ \mathfrak{g}^{F_{\mathrm{lim}}}
\colonequals \bigoplus_{p\geq0} \mathfrak{g}^{p,q}, \quad \mathfrak{q} \colonequals  \bigoplus_{p< 0} \mathfrak{g}^{p,q}.\]
Then we have $\mathfrak{g}=\mathfrak{g}^{F_{\mathrm{lim}}}\oplus \mathfrak{q}$ and $\mathfrak{q}$ is a nilpotent Lie subalgebra of $\mathfrak{g}$. Then
we can write $\Psi(s)=e^{\Gamma(s)}F_{\mathrm{lim}}$ for a unique holomorphic map $\Gamma:\Delta^n \to \mathfrak{q}$ with $\Gamma(0)=0$. Therefore we can put the period map into the following standard form
\begin{equation}\label{eqn: normal form of period map}
\Phi(z)=e^{z_1N_1}e^{\Gamma(s)}F_{\mathrm{lim}}=e^{X(z)}F_{\mathrm{lim}},
\end{equation}
with $X(z)\in \mathfrak{q}$ being nilpotent.

Let us recall a standard fact about mixed Hodge structures.
\begin{fact}\label{claim: lowest Hodge intersect with weight}
Let $(V,F^{\bullet}V,W_{\bullet}V)$ be a mixed Hodge structure of weight $k$ where 
\[ W_{\bullet}V: \{0\}=W_{0}V\subseteq W_{1}V\cdots \subseteq W_{2k}V=V \]
 is induced by a nilpotent operator $N:V\to V$ with $N^{\ell}:\gr^W_{k+\ell}V\xrightarrow{\sim} \gr^W_{k-\ell}V$ and 
\[ N:F^{\bullet}V\to F^{\bullet-1}V\]
is strict with respect to the Hodge filtration.
 Let $F^qV$ be the smallest nonzero piece in the Hodge filtration, then
\[ F^qV \cap W_{k-1}V=\{0\}.\]
\end{fact}

\begin{proof}
An immediate consequence of the isomorphism $N^i: \gr^W_{k+i}V \xrightarrow{\sim} \gr^W_{k-i}V$ is 
\[ W_{k-1}V \subseteq \mathrm{Im}(N).\]
Using the assumption that $N$ is strict with respect to the Hodge filtration $F^{\bullet}$, we conclude that
\[ F^qV \cap W_{k-1}V \subseteq F^qV \cap \mathrm{Im}(N)=\mathrm{Im}(N|_{F^{q+1}V})=\{0\}. \]
\end{proof}
Suppose the VHS $\cV$ has weight $k$, then the mixed Hodge structure $(V,W,F_{\mathrm{lim}})$ has weight $k$ (as in \cite{CKS}) and satisfies the assumption of Fact \ref{claim: lowest Hodge intersect with weight}. Let $F^q_{\mathrm{lim}}$ be the smallest non-zero piece, then by Fact \ref{claim: lowest Hodge intersect with weight} we have
\[ F^q_{\mathrm{lim}}V \cap W_{k-1}V=\{0\}.\]
In particular, the decomposition (\ref{eqn:decomposition of W and F}) becomes
\begin{equation}\label{eqn:refined decomposition}
F^q_{\mathrm{lim}}=\bigoplus_{a\geq q, a+b\geq k} I^{a,b}.
\end{equation}

Now we choose a $Q$-orthogonal basis of each $I^{a,b}$, which gives a grading $V\cong \oplus_{\ell \in \mathbb{Z}} \mathrm{Gr}^W_{\ell}V$ by (\ref{eqn:decomposition of W and F}). Suppose $v\in F^q_{\mathrm{lim}}$, using the standard form of the period map (\ref{eqn: normal form of period map}) we see that
\[ \sigma\colonequals e^{X(z)}v\]
gives a section of the smallest piece  $F^q\cV^{>-1}$. Therefore a basis of $F^q_{\mathrm{lim}}V$ gives rise to a frame of $F^q\cV^{>-1}$ and it suffices to prove the norm estimates for these frame sections. We will use $A\sim B$ to mean that $A$ and $B$ are mutually bounded up to multiplication by a constant. By Schmid's Theorem \ref{thm:one dimension hodge metric}, we have
\[ |e^{X(z)}v|^2_{\Phi(z)} \sim |e^{X(z)}v|^2_{Q}\]
If $v\in I^{a,b}$, then one can see that the dominant term of $|e^{X(z)}v|^2_{Q}$ is
\[ (-\log |s_1|)^{a+b-k}Q(v,N_1^{a+b-k}v)=(-\log |s_1|)^{a+b-k}|v|^2_{Q_{a+b-k}}\]
Using $v\in F^q_{\mathrm{lim}}$ and (\ref{eqn:refined decomposition}), we have
\[ |e^{X(z)}v|^2_{Q} \sim  \sum_{\ell \geq 0} (-\log |s_1|)^{\ell} |v|^2_{Q_\ell}\]
Now suppose $\sigma$ is an arbitrary section of $F^q\cV^{>-1}$, and we write as $\sigma=\sum f_{\alpha}\sigma_{\alpha}$, where $\{\sigma_{\alpha}\}$ is determined by a basis of $F^q_{\mathrm{lim}}V$ as above. When one expands $|\sum f_{\alpha}\sigma_{\alpha}|^2_{h}$, using the estimates above, we see that the leading term is
\[ \sum_{\alpha} |f_{\alpha}|^2(-\log |s_1|)^{n_{\alpha}}, \quad n_{\alpha}\geq 0.\]

\end{proof}

\subsection{Hodge modules}\label{section:HM}
Let $X$ be a complex manifold of dimension $n$. Let $\mathrm{HM}^p(X,w)$ be the category of polarized Hodge modules of weight $w$ on $X$ with strict support $X$. One of the main results in Saito's theory of Hodge modules \cite[Theorem 3.21]{Saito90a} is 
\begin{thm}[Structure Theorem]\label{thm:structure}
 Let X be a complex manifold of dimension $n$.
\begin{enumerate}
\item If $\cV$ is a polarized variation of rational Hodge structures of weight $(w - n)$ on a dense Zariski open subset of $X$, then $\cV$ extends uniquely to an object of $\mathrm{HM}^p(X,w)$.
\item Conversely, every object $\cM$ of $\mathrm{HM}^p(X,w)$ is obtained this way.
\end{enumerate}
\end{thm}

Assume $X=\D^n$ and $U$ is a Zariski open subset of $X$ such that $D \colonequals X\setminus U$ is a divisor (If $X\setminus U$ is of higher codimension, the VHS on $U$ automatically extends to a VHS on $X$). Let $j: U \hookrightarrow X$ be the open embedding. Let $\cV$ be a VHS over $U$ and let $\cM$ be the polarized Hodge module on $X$ with strict support $X$ corresponding to $\cV$. We would like to recall the construction of the Hodge filtration $F_\bullet \cM$ on $\cM$, which consists of two steps.

\textbf{Step 1.} Assume $D$ is a simple normal crossing divisor, i.e. $D=\cup D_j$ and we can choose coordinates $s_1,\ldots,s_n$ on $X$ so that $D_j\colonequals \{ s_j = 0\}$. Then $U=(\D^{\ast})^n$. Denote $\cV^{> -1}$ to be the Deligne's canonical lattice associated to $\cV$ as in Definition \ref{definition:canonical lattice}. Schmid's nilpotent orbit theorem \cite{Schmid}  guarantees that the Hodge filtration on $\cV$ extends to a Hodge filtration on Deligne's canonical lattice\footnote{For VHS with unipotent monodromy, this is proved by \cite[Theorem 4.13]{Schmid}, as a consequence of \cite[Theorem 4.12]{Schmid}. Saito explains in \cite[\S 3.10]{Saito90a} how to deduce the statement for VHS with quasi-unipotent monodromy from Schmid's results.}, i.e.
\[ F^k(\cV^{>-1}) \colonequals  j_{\ast}(F^k\cV) \cap \cV^{>-1}\]
are holomorphic subbundles of $\cV^{>-1}$. We need to reindex this decreasing filtration to an increasing filtration for $D$-modules: 
\[ F_k(\cV^{>-1}) \colonequals F^{-k}(\cV^{>-1}).\]
It is proved in \cite[(3.18.2)]{Saito90a} that $\cM$ is generated by $\cV^{>-1}$ as a $D_X$-module:
\[ \cM = D_X \cdot \cV^{>-1}.\]
Let $i=(i_1,\ldots,i_n) \in \mathbb{Z}_{+}^{n}$ be a $n$-tuple of positive integers. We denote $\abs{i} \colonequals i_1+\cdots+i_n$ and $\partial^{i}_s \colonequals \partial^{i_1}_{s_1}\cdots \partial^{i_n}_{s_n}$. The increasing Hodge filtration $ F_{\bullet}\cM$ is defined as
\begin{eqnarray*}
F_k\cM &\colonequals& \sum_{i\in \mathbb{Z}_{+}^{n}} \partial^{i}_s(F_{k-\abs{i}}\cV^{>-1})\\
&=& \sum_{i\in \mathbb{Z}_{+}^{n}} \partial^{i}_s(F^{\abs{i}-k}\cV^{>-1}) \\
&=& \sum_{i\in \mathbb{Z}_{+}^{n}} \partial^{i}_s(j_{\ast}(F^{\abs{i}-k}\cV)\cap \cV^{>-1}).
\end{eqnarray*}
Let $q \colonequals \max \{ k : F^{k}\cV \neq 0 \}$ to be the lowest nonzero piece in the Hodge filtration of $\cV$ over $U$.  Then the lowest nonzero piece in the Hodge filtration of $\cM$ is
\begin{equation}\label{eqn:filtration}
F_{-q}\cM = j_{\ast}(F^{q}\cV) \cap \cV^{>-1}. 
\end{equation}

\begin{remark}In \cite{Saito90b}, Saito has the formula for right Hodge module:
\[ F_k\cM = \sum_{i\in \mathbb{Z}_{+}^{n}} (j_{\ast}(F^{\abs{i}-k-n}\cV) \cap \cV^{>-1})\partial^{i}_s.\]
Here we convert it to the filtration on the corresponding left $D$-module, which has a shift by $n=\dim X$.
\end{remark}

\textbf{Step 2.} Assume $D$ is an arbitrary divisor. Let $f: (\tilde{X},\tilde{D}) \to (X,D)$ be a log resolution such that $\tilde{D}$ is a simple normal crossing divisor and $\tilde{X}\setminus \tilde{D} \cong X\setminus D$. Denote $\tilde{U} \colonequals \tilde{X}\setminus \tilde{D}$. Consider $\tilde{\cV}=(f|_{\tilde{U}})^{\ast}\cV$ to be the VHS on $\tilde{U}$. Let $\tilde{\cM}$ be the Hodge module on $\tilde{X}$ associated to $\tilde{\cV}$. Let $F_p\tilde{\cM}$ and $F_p\cM$ be the lowest nonzero piece in the Hodge filtrations. 
As a corollary of the direct image theorem, Saito \cite[Theorem 1.1]{Saito90b} proved that
\begin{eqnarray}\label{eqn:direct image}
f_{\ast}(F_p\tilde{\cM}\otimes \omega_{\tilde{X}}) = F_p\cM\otimes \omega_X.
\end{eqnarray}
The entire filtration $F_{\bullet}\cM$ is defined in a more complicated way. But for our purpose, it is enough to understand the lowest nonzero piece.

\subsection{Nakano positivity of Hodge bundles}\label{section:nakano}
Let $X$ be a complex manifold. Let $E$ be a holomorphic vector bundle on $X$ with a hermitian metric $h$, we say that $h$ is \textit{Nakano semi-positive} if the curvature tensor $\Theta_h$ is semi-positive definite as a hermitian form on $T_X\otimes E$, i.e. if for every $u\in \Gamma(T_X\otimes E)$, we have $\sqrt{-1}\Theta_h(u,u)\geq 0$ (see \cite{Demailly}).

\begin{ex}\label{ex:Schmid}
Let $\cV$ be a VHS on $X$ and let $F^q\cV$ be the lowest nonzero piece in the Hodge filtration. Recall that the Hodge metric $h$ on $F^q\cV$ is defined as follows: if $v,w \in H^0(X,F^q\cV)$, then
 \[ h(v,w) \colonequals (-1)^qS^h(v,\bar{w}).\]
By Schmid's curvature calculation \cite[Lemma 7.18]{Schmid}, $h$ is Nakano semi-positive.
\end{ex}
 
Nakano semi-positive bundles satisfy an optimal version of $L^2$ extension property proved by Blocki and Guan-Zhou \cite{Blocki,GZ}. Let $s_1,\ldots,s_n$ be the coordinates on $\mathbb{C}^n$, we write the standard Lebesgue measure to be $d\mu \colonequals c_n ds_1\wedge d\bar{s}_1\wedge \cdots \wedge ds_n\wedge d\bar{s}_n$ and $c_n=2^{-n}(-1)^{n^2/2}$.

\begin{thm}[Sharp Ohsawa-Takegoshi]\label{thm:Guan-Zhou}
Let $B\subset \mathbb{C}^n$ be the open unit ball. Let $Z$ be an analytic subset of $B\setminus \{0\}$. Let $(E,h)$ be a holomorphic vector bundle over $B\setminus Z$ such that $h$ is Nakano semi-positive. Then for every $v\in E_0$ with $\abs{v}_{h,0}=1$, there is a holomorphic section $\sigma \in H^0(B\setminus Z, E)$
with 
\[ \sigma(0)=v \text{ and } \frac{1}{\mu(B)}\int_B |\sigma|^2_h \ d\mu \leq 1.\]
\end{thm}

\begin{remark}
It is easy to see that $(B\setminus Z, \{0\})$ satisfies the condition (ab) in \cite[Definition 1.1]{GZ}. Then we can apply \cite[Theorem 2.2]{GZ} to the pair $(B\setminus Z, \{0\})$, taking $A=0, c_A(t)\equiv 1$ and $\Psi=n\log(\abs{s_1}^2+\cdots +\abs{s_n}^2)$. Guan-Zhou's definition of $d\mu$ doesn't involve the constant $c_n$, but it reduces to the version we write here because the two $c_n$'s in the inequality get cancelled.
\end{remark}

\subsection{Singular hermitian metrics and metric positivity}\label{section:metric positivity}
In this section, we recall the notion of singular hermitian metrics on torsion-free sheaves from \cite{HPS,PT}. Let $X$ be a complex manifold. Let $E$ be a holomorphic vector bundle on $X$ of some rank $r\geq 1$.

\begin{definition}
A \emph{singular hermitian metric} on $E$ is a function $h$ that associates to every point $x\in X$ a singular hermitian inner product $\abs{-}_{h,x}:E_x \to [0,+\infty]$ on the complex vector space $E_x$, subject to the following two conditions:
\begin{enumerate}

\item $h$ is finite and positive definite almost everywhere, meaning that for all $x$ outside a set of measure zero, $\abs{-}_{h,x}$ is a hermitian inner product on $E_x$.

\item $h$ is measurable, meaning that the function 
\[ \abs{s}_h: U \to [0,+\infty], \quad x\mapsto \abs{s(x)}_{h,x},\]
is measurable whenever $U\subseteq X$ is open and $s\in H^0(U,E)$.
\end{enumerate}
\end{definition}

\begin{definition}
Let $h$ be a singular hermitian metric on a holomorphic bundle $E$. We say that the pair $(E,h)$ has \emph{semi-positive curvature} if the function $\log \abs{f}_{h^{\ast}}$ is plurisubharmonic for every $f\in H^0(U,E^{\ast})$ and every open subset $U\subseteq X$, where $h^{\ast}$ is the induced singular hermitian metric on the dual bundle $E^{\ast}$.
\end{definition}
Let $\mathcal{F}$ be a torsion-free coherent sheaf on $X$. Let $X(\mathcal{F})$ denote the maximal open subset where $\mathcal{F}$ is locally free; then $X\setminus X(\mathcal{F})$ is a closed analytic subset of codimension $\geq 2$. In particular, the coherent sheaf $E\colonequals \mathcal{F}|_{X(\mathcal{F})}$ is locally free.

\begin{definition}
A \textit{singular hermitian metric} on $\mathcal{F}$ is a singular hermitian metric $h$ on the holomorphic vector bundle $E=\mathcal{F}|_{X(\mathcal{F})}$. We say that $h$ has \emph{semi-positive curvature} if the pair $(E,h)$ has semi-positive curvature.
\end{definition}
Inspired by the sharp $L^2$ extension Theorem \ref{thm:Guan-Zhou}, it is natural to consider the following ``minimal extension property" for singular hermitian metrics on torsion-free sheaves \cite{HPS}. Let $B\subset \mathbb{C}^n$ be the open unit ball.

\begin{definition}\label{MEP}
We say that a singular hermitian metric $h$ on $\mathcal{F}$ has the \textit{minimal extension property} if there exists a nowhere dense closed analytic subset $Z \subset X$ with the following two properties
\begin{enumerate}
\item $\mathcal{F}$ is locally free on $X\setminus Z$.

\item For every embedding $\iota: B \hookrightarrow X$ with $x=\iota(0) \in X\setminus Z$, and every $v\in E_x$ with $|v|_{h,x}=1$, there is a holomorphic section $\sigma \in H^0(B,\iota^{\ast}\mathcal{F})$ such that 
\[ \sigma(0)=v \text{ and } \frac{1}{\mu(B)}\int_B |\sigma|^2_h \ d\mu \leq 1.\]
\end{enumerate}

\end{definition}
For consequences and applications of minimal extension properties, readers may refer to \cite[\S 26]{HPS}. 

\section{Proofs}\label{section:proof}

In this section, we will work with the following set up:

\begin{setup}\label{setup for the proof} -
\begin{itemize}
\item Let $X$ be a complex manifold of dimension $n$ and let $Z$ be a closed analytic subset of $X$. 

\item Let $\cV$ be a VHS over $U\colonequals X\setminus Z$ and let $\cM$ be the polarized Hodge module on $X$ with strict support $X$ corresponding to $\cV$ via Saito's structure theorem~\ref{thm:structure}. Since any VHS away from codimension 2 automatically extends, we can assume that $Z=D$ is a divisor on $X$. 

\item Let $F^q\cV$ be the lowest nonzero piece in the Hodge filtration of $\cV$ and $h$ be the Hodge metric on $F^q\cV$.
\end{itemize}
\end{setup}
By \S \ref{section:HM}, $F_{-q}\cM$ is the lowest nonzero piece in the Hodge filtration of $\cM$ and 
\begin{eqnarray*}
F_{-q}\cM\big|_U= F^q\cV.
\end{eqnarray*} 
Moreover, $F_{-q}\cM$ is a torsion-free coherent sheaf by \cite[Proposition 2.6]{Saito90b}.

\subsection{Proof of Corollary \ref{cor:minimal extension property}}
In this section, we prove Corollary \ref{cor:minimal extension property} assuming Theorem \ref{thm:main}. 
 
First we show that $h$ extends to a singular hermitian metric on $F_{-q}\cM$ with
semi-positive curvature. Let $h^{\ast}$ denote the induced metric on the dual bundle
$(F^q \cV)^{\ast}$. It is enough to prove that $h^{\ast}$ extends to a singular
hermitian metric on the dual sheaf $(F_{-q}\cM)^{\ast}$ with semi-negative curvature,
since the two metrics determine each other by \cite[Prop.~17.2]{HPS}. Let $g$ be a
nontrivial holomorphic section of the dual sheaf $(F_{-q}\cM)^{\ast}$, defined over
some open subset $V \subseteq X$. Since the Hodge metric on $F^q \cV$ is Nakano
semi-positive (in particular Griffiths semi-positive) by Example~\ref{ex:Schmid}, the
function 
\[ 
	\psi\colonequals \log |g|_{h^{\ast}}: V\cap U \to (-\infty,+\infty)
\] 
is pluri-subharmonic on $V \cap U$. To show that $h^{\ast}$ extends to a singular
metric, it suffices to show that $\psi$ extends to a pluri-subharmonic function on
all of $V$. For that, we only need to prove that $\psi$ is bounded from above near any
point $x \in (D \setminus D_{\mathrm{sing}}) \cap V$; we can then apply the Riemann
extension theorem, followed by Hartog's theorem (because $D_{\mathrm{sing}} \cap V$
has codimension $\geq 2$).

Since $D \setminus D_{\mathrm{sing}}$ is smooth, the restriction of $F_{-q}\cM$ to
the open set $(D \setminus D_{\mathrm{sing}}) \cap V$ is isomorphic to $F^q
\cV^{>-1}$, where $\cV^{>-1}$ is Deligne's canonical extension. In suitable local
coordinates $s_1, \dotsc, s_n$ centered at the point $x$, the divisor $D \setminus
D_{\mathrm{sing}}$ is defined by $s_1 = 0$. Using Lemma \ref{lemma:refined version of
Schmids Theorem}, we can find a local holomorphic frame
$\{\sigma_{\alpha}\}_{\alpha}$ for the bundle $F^q \cV^{>-1}$ such that the pointwise
Hodge norm squared of any section $\sigma = \sum_{\alpha} f_{\alpha} \sigma_{\alpha}$
is uniformly bounded from below by
\[ 
	|\sigma|_h^2 \geq C \sum_{\alpha} |f_{\alpha}|^2,
\]
where $C > 0$ is a positive constant; the point is that the exponent $n_{\alpha}$ of the $(-\log |s_1|)$ is $ \geq 0$.
Consequently, 
\[
	\frac{|g(\sigma)|^2}{|\sigma|_h^2} \leq 
	\frac{\sum_{\alpha} |g(\sigma_{\alpha})|^2 \sum_{\alpha} |f_{\alpha}|^2}%
	{C \sum_{\alpha} |f_{\alpha}|^2}
	\leq C^{-1} \sum_{\alpha} |g(\sigma_{\alpha})|^2,
\]
and because $g$ is a holomorphic section of the dual sheaf, the expression on the
right-hand side is bounded independently of $\sigma$. Since $\sigma$ is an arbitrary
holomorphic section of $F^q \cV^{>-1}$, it follows that $|g|_{h^{\ast}}^2$, and
therefore $\psi$, is bounded from above near $x$, as claimed.

Next we show that $h$ has the minimal extension property in the sense of Definition \ref{MEP}. 
Note that $D$, being a divisor on $X$, automatically satisfies the condition (1) in Definition~\ref{MEP}. For condition (2), let us fix an embedding $\iota: B \hookrightarrow X$ with $x=\iota(0) \in X\setminus D$. Choose $v\in (F^q\cV)_x$ with $|v|_{h,x}=1$. Since $h$ is Nakano semi-positive over $B \setminus D$, Theorem \ref{thm:Guan-Zhou} implies that there exists a holomorphic section $\sigma \in H^0(B\setminus D,F^q\cV|_{B\setminus D})$ satisfying
\[ \sigma(x)=v \quad \text{and}  \quad \frac{1}{\mu(B)}\int_{B} |\sigma|^2_h \ d\mu \leq 1. \]
Now we can apply Theorem \ref{thm:main} to conclude that $\sigma$ extends to a holomorphic section in $H^0(B, \iota^\ast(F_{-q}\cM))$. Therefore, $h$ has the minimal extension property.

\subsection{Proof of Theorem \ref{thm:main}: first half}\label{sec:first half}
Because Theorem \ref{thm:main} is a local statement, we can assume $X=B$ is a ball in $\mathbb{C}^n$ and $D\subseteq X$ is an arbitrary divisor. Denote $j:B\setminus D \hookrightarrow B$ to be the open embedding. We would like to prove the first half of Theorem \ref{thm:main}: let 
\[ \sigma \in H^0(B\setminus D,F^q\cV|_{B\setminus D})\] be a holomorphic section locally $L^2$ near $D$, then 
\[ \sigma \in H^0(B, F_{-q}\cM).\]

\textbf{Step 1}. We reduce to the situation where $D$ is a simple normal crossing divisor. Let 
\[ f: (\tilde{B},\tilde{D}) \to (B,D) \]
be a log resolution such that $\tilde{B}\setminus \tilde{D} \cong B\setminus D$ and $\tilde{D}$ is a simple normal crossing divisor. The VHS $\cV$ pulls back isomorphically to a VHS on $\tilde{B}\setminus \tilde{D}$ which is denoted by $\tilde{\cV}$. Abusing notations, we use $h$ to denote the Hodge metric on $\tilde{\cV}$. Let $\tilde{\cM}$ be the Hodge module on $\tilde{B}$ uniquely determined by $\tilde{\cV}$ via Saito's Structure Theorem \ref{thm:structure}.

Fixing coordinates $s_1,\ldots,s_n$ on $B$, then $\beta\colonequals \sigma\otimes ds_1\wedge \cdots \wedge ds_n$ is a section of $H^0(B\setminus D,F^q\cV\otimes \omega_{B\setminus D})$. Define $\beta \wedge \overline{\beta} \colonequals \abs{\sigma}^2_h d\mu$, the $L^2$ condition on $\sigma$ implies that \begin{eqnarray}\label{eqn:pull back}
\int_{B\setminus D} \beta \wedge \overline{\beta}< \infty.
\end{eqnarray}
Since $f$ induces an isomorphism $\tilde{B}\setminus \tilde{D} \cong B\setminus D$, we can view $\beta$ as a section of $H^0(\tilde{B}\setminus \tilde{D},F^q\tilde{\cV}\otimes \omega_{\tilde{B}\setminus\tilde{D}})$.
In Step 2, we will show that the integrability condition \eqref{eqn:pull back} implies that 
\[ \beta \in H^0(\tilde{B},F_{-q}\tilde{\cM}\otimes \omega_{\tilde{B}}).\] 
By the direct image theorem, the identification (\ref{eqn:direct image}) says 
\[ H^0(\tilde{B},F_{-q}\tilde{\cM}\otimes \omega_{\tilde{B}})=H^0(B,F_{-q}\cM\otimes \omega_B). \]
We conclude that $\beta \in H^0(B,F_{-q}\cM\otimes \omega_B)$. In particular, 
\[ \sigma \in H^0(B,F_{-q}\cM).\]

\textbf{Step 2}. Denote $\tilde{j}:\tilde{B}\setminus \tilde{D} \hookrightarrow \tilde{B}$ to be the open embedding. Working locally, we may replace $\tilde{B}$ by $\D^{n}$ and $\tilde{D}$ by $\D^{n}\setminus (\D^{\ast})^n$. Let $s_1,\ldots,s_n$ be coordinates on $\D^n$. We view $\beta$ as a section over $\tilde{B}\setminus \tilde{D}$ and write $\beta=\tilde{\sigma} \otimes ds_1 \wedge \ldots \wedge ds_n$. Here $\tilde{\sigma}\colonequals \sigma$ is viewed as a section in $H^0((\Delta^{\ast})^n,F^q\tilde{\cV})$. The condition (\ref{eqn:pull back}) implies that
\[ \int_{(\D^{\ast})^n} |\tilde{\sigma} |^2_h \ d\mu < \infty, \]
Thus $\tilde{\sigma} \in H^0(\D^n,\tilde{\cV}^{> -1})$ by Proposition \ref{prop:integrable sections} (which will be proved below).
Over $\D^{n}$, by (\ref{eqn:filtration}) we have 
\[ F_{-q}\tilde{\cM}=\tilde{j}_{\ast}(F^{-q}\tilde{\cV})\cap \tilde{\cV}^{>-1}, \] 
then $\tilde{\sigma} \in H^0(\D^n, F_{-q}\tilde{\cM})$. Therefore,
\[ \beta \in H^0(\D^n,F_{-q}\tilde{\cM}\otimes \omega_{\D^n}).\]

\subsection{$L^2$ sections are holomorphic}
In this section, we give the main technical result:

\begin{prop}\label{prop:integrable sections}
Let $\cV$ be a VHS over $(\D^{\ast})^n$, let $F^q\cV$ be the lowest nonzero piece in the Hodge filtration, and let $h$ be the Hodge metric on $F^q\cV$.
If a section $\sigma \in H^0((\D^{\ast})^n,F^q\cV)$ satisfies
\[ \int_{(\D^{\ast})^n} |\sigma |^2_h \ d\mu < \infty, \]
then ${\sigma} \in H^0(\D^n,\cV^{> -1})$, where $j:(\D^{\ast})^n \to \D^n$ is the open embedding.
\end{prop}

We will first prove a reduction lemma.
\begin{lemma}\label{lem:reduction}
If the Proposition \ref{prop:integrable sections} holds for VHS with unipotent monodromy, then it holds for VHS with quasi-unipotent monodromy.
\end{lemma}

\begin{proof}
Let $\cV$ be a VHS with quasi-unipotent monodromy over $(\D^{\ast})^n$ with $\rank\cV=r$. By Hartog's theorem, to show any given section $\sigma$ with the integrability condition lies in $H^0(\D^n,\cV^{>-1})$, it suffices to show that $\sigma$ extends over any generic point of each divisor $D_j$. Without loss of generality, we only need to prove this for $D_1$. Around a generic point of $D_1$, we can choose a coordinate neighborhood $U_1$ such that it is isomorphic to $\D^{\ast} \times \D^{n-1}$. Let $T$ be the monodromy operator of $\cV|_{U_1}$ with Jordan decomposition $T=T_sT_u$. Let $m$ be the quasi-unipotency index of $T$ such that $T_s^{m}=\mathrm{Id}$. 

Let $p$ be the universal covering map of $\D^{\ast}\times \Delta^{n-1}$:
\[ p: \mathbb{H} \times \Delta^{n-1} \to \D^{\ast} \times \Delta^{n-1}, \quad (z_1,z_2,\ldots,z_n) \mapsto (e^{2\pi \sqrt{-1}z_1},z_2,\ldots,z_n).\]
Set $V\colonequals H^0(\mathbb{H}\times \Delta^{n-1},p^{\ast}\cV)^{p^\ast \nabla}$, which is isomorphic to a fiber of $\cV$ as vector spaces. Choose a basis $(v_\alpha)_{1\leq \alpha \leq r}$ of $V$ which diagonalizes $T_s$. We choose integers $0\leq k_{\alpha} \leq m-1$ such that 
\[ T_s \cdot v_{\alpha} = e^{-2\pi \sqrt{-1}k_{\alpha }/m}\cdot v_{\alpha}.  \]
By Construction~\ref{def:log}, let $R \colonequals \log T$ with the decomposition 
\[ R=H + N, \]
where $H=\log T_s$ and $N=\log T_u$. By the construction we have
\[ H \cdot v_{\alpha} = -(2\pi \sqrt{-1}k_{\alpha}/m)\cdot v_{\alpha}, \]
and $N$ is nilpotent.

To reduce to the unipotent situation, consider the following unramified covering map: 
\[ f: \D^{\ast} \times \Delta^{n-1} \to \D^{\ast}\times \Delta^{n-1}, (\tilde{s}_1,\tilde{s}_2,\ldots,\tilde{s}_n) \mapsto (\tilde{s}_1^{m},\tilde{s}_2,\ldots,\tilde{s}_n).\]
Let $(\tilde{\cV},\tilde{\nabla}) \colonequals f^{\ast}(\cV,\nabla)$ to be the VHS polarized by $f^{\ast}h$. Note that $(\tilde{\cV}, \tilde{\nabla})$ has unipotent monodromy because the monodromy operator $\tilde{T}$ equals to $T^{m}=T_u^{m}$. Set $\tilde{R} \colonequals \log \tilde{T}$, then
$\tilde{R}=mN$. The finite base change comes with the following diagram:
\[ \begin{CD}
\mathbb{H}\times \Delta^{n-1} @>\tilde{f}>> \mathbb{H}\times \Delta^{n-1} \\
@V\tilde{p}VV             @V{p}VV \\
\D^{\ast}\times \Delta^{n-1} @>{f}>>   \D^{\ast}\times \Delta^{n-1}
\end{CD}
\]
where 
\begin{eqnarray*}
\tilde{p}(\tilde{z}_1,\tilde{z}_2,\ldots, \tilde{z}_n)=(e^{2\pi\sqrt{-1}\tilde{z}_1},\tilde{z}_2,\ldots, \tilde{z}_n),\\
\tilde{f}(\tilde{z}_1,\ldots,\tilde{z}_n)=(m\tilde{z}_1,\tilde{z}_2,\ldots,\tilde{z}_n).
\end{eqnarray*}
Let $\tilde{V}$ to be $H^0(\mathbb{H}\times \Delta^{n-1},\tilde{p}^\ast\tilde{\cV})^{\tilde{p}^\ast \tilde{\nabla}}$.
Then as in section \ref{section:can},
\[ \cV^{>-1} \cong (e^{z_1R}V)\otimes_{\mathbb{C}} \mathcal{O}_{\D^n},  \quad  \tilde{\cV}^{>-1} \cong (e^{\tilde{z}_1\tilde{R}}\tilde{V})\otimes_{\mathbb{C}} \mathcal{O}_{\D^n}.\]

Now we want to calculate the pull back of generating sections of $\cV^{>-1}$ via $f$. Since they are identified with holomorphic functions on $\mathbb{H}\times \Delta^{n-1}$, it suffices to pull back via $\tilde{f}$. Let $v\in V$ and $w=e^{z_1R}v$, then 
\begin{eqnarray*} 
\tilde{f}^{\ast}(w) &=&e^{\tilde{z}_1mR}\tilde{f}^{\ast}v = e^{\tilde{z}_1mH}\cdot e^{\tilde{z}_1mN}\tilde{f}^{\ast}v \\
&=& e^{\tilde{z}_1mH}\cdot e^{\tilde{z}_1\tilde{R}}\tilde{f}^{\ast}v = e^{\tilde{z}_1mH}\tilde{w}.
\end{eqnarray*}
where $\tilde{w}=e^{\tilde{z}_1\tilde{R}}\tilde{f}^{\ast}v \in H^0(\Delta^{\ast}\times \D^{n-1},\tilde{\cV}^{>-1})$.

Since $(v_{\alpha})_{1\leq \alpha \leq r}$ is the basis of $V$ such that
\begin{eqnarray*} 
H \cdot v_{\alpha} = -(2\pi \sqrt{-1}k_{\alpha}/m)\cdot v_{\alpha}.
\end{eqnarray*}
Set $w_{\alpha}=e^{z_1R}v_{\alpha}$, then
\begin{eqnarray}\label{eqn:pullback canonical extension}
\tilde{f}^{\ast}(w_{\alpha}) = e^{\tilde{z}_1mH}\tilde{w}_{\alpha}= e^{-2\pi\sqrt{-1}\tilde{z}_1k_{\alpha }}\tilde{w}_{\alpha}= \tilde{s}_1^{-k_{\alpha }} \tilde{w}_{\alpha}.
\end{eqnarray}
where $\tilde{w}_{\alpha}=e^{\tilde{z}_1\tilde{R}}\tilde{f}^{\ast}v_{\alpha}$. 
\begin{remark}
The generating sections of $\cV^{>-1}$ pull back to generating sections of $\tilde{\cV}^{>-1}$ with an extra monomial factor.
\end{remark}

Now we start the proof of Lemma \ref{lem:reduction}. Let $\sigma$ be a section of $F^q\cV$ over $\Delta^{\ast}\times \Delta^{n-1}$ such that $\int_{\Delta^{\ast}\times \Delta^{n-1}} |\sigma|^2_h \ d\mu < +\infty$. By the construction of Deligne's canonical lattice, 
\[ H^0\left(\Delta^{\ast}\times \Delta^{n-1}, F^q\cV \right) \subset H^0\left(\Delta^{\ast}\times \Delta^{n-1}, \cV^{>-1}|_{\Delta^{\ast}\times \Delta^{n-1}}\right). \]
Therefore we can write 
\[ \sigma= \sum_{\alpha = 1}^{r} h_{\alpha}w_{\alpha}, \] 
where $h_{\alpha}$ is a holomorphic function over $\Delta^{\ast}\times \Delta^{n-1}$. To show $\sigma \in H^0(\D^{n},\cV^{>-1})$, it suffices to show that each $h_\alpha$ is actually a holomorphic function over $\D^n$.

Since $\int_{\Delta^{\ast}\times \Delta^{n-1}} \abs{\sigma}^2_h \ d\mu < +\infty$, local calculation shows that 
\[ \int_{\Delta^{\ast}\times \Delta^{n-1}} f^{\ast}(\abs{\sigma}^2_h \ d\mu) = \int_{\Delta^{\ast}\times \Delta^{n-1}} \big|m\tilde{s}_1^{m-1} \cdot f^\ast(\sigma)\big|^2_{f^\ast h} d\tilde{\mu} < +\infty.\]
Here 
\[ d\mu=c_n ds_1\wedge d\bar{s}_1\wedge \cdots \wedge ds_n\wedge d\bar{s}_n,  d\tilde{\mu}=c_n d\tilde{s}_1\wedge d\bar{\tilde{s}}_1\wedge \cdots \wedge d\tilde{s}_n\wedge d\bar{\tilde{s}}_n\]
are the standard Lebesgue measure on $\D^n$ and $c_n=2^{-n}(-1)^{n^2/2}$.
Since we assume that Proposition \ref{prop:integrable sections} is true for VHS with unipotent monodromy, it follows that
\[ \tilde{s}_1^{m-1} \cdot f^\ast(\sigma)  \]
must belongs to $H^0(\D^n, \tilde{\cV}^{>-1})$, which equals to $H^0(\D^n,\tilde{\cV}^0)$ by remark \ref{remark:unipotent}.
By (\ref{eqn:pullback canonical extension}),
\begin{eqnarray*}
f^\ast(\sigma) &=& \sum_{\alpha = 1}^{r} f^{\ast}(h_{\alpha}w_{\alpha}) \\
             &=& \sum_{\alpha= 1}^{r}  f^{\ast}(h_{\alpha}) \cdot \tilde{s}_1^{-k_{\alpha}}\tilde{w}_{\alpha}
\end{eqnarray*}
Hence
\[  \tilde{s}_1^{m-1} \cdot f^\ast(\sigma)
= \sum_{\alpha= 1}^{r} \left(f^{\ast}(h_{\alpha}) \cdot \tilde{s}_1^{m-1-k_{\alpha }}\right)\tilde{w}_{\alpha}. \]
Therefore we conclude that for each $1\leq \alpha \leq r$, 
the corresponding coefficient function
\[ f^{\ast}(h_{\alpha}) \cdot \tilde{s}_1^{m-1-k_{\alpha }} \]
is holomorphic in $\tilde{s}_1,\ldots,\tilde{s}_n$. In particular, this implies that $h_{\alpha}$ must be a meromorphic function. For each $\alpha$, we let $n_{\alpha }$ be the lowest power of $s_1$ in the Laurent expansion of $h_{\alpha}$, then the lowest power of $\tilde{s}_1$ in $f^{\ast}(h_{\alpha})$ is $mn_{\alpha}$. Holomorphicity imply that
\[  mn_{\alpha }+m-1-k_{\alpha } \geq 0. \]
Because $k_{\alpha }\geq 0$, we have
\[ n_{\alpha} \geq -1+\frac{1+k_{\alpha }}{m} \geq  -1+\frac{1}{m}, \]
Since ${n_{\alpha }}$ is an integer, it follows that $n_{\alpha }\geq 0$, i.e. $h_{\alpha}$ is actually a holomorphic function in $s_1,\ldots,s_n$. Therefore
we conclude that $\sigma \in H^0(\D^n,\cV^{>-1})$.

\end{proof}

\begin{proof}[Proof of Proposition \ref{prop:integrable sections} ]

Let $\cV$ be a VHS over $(\D^{\ast})^n$ and $D\colonequals \D^n\setminus (\D^{\ast})^n=\cup D_j$ is a simple normal crossing divisor. By Borel's monodromy theorem \cite[Lemma 4.5]{Schmid}, $(\cV,\nabla)$ has quasi-unipotent monodromy. Then we can assume that $(\cV,\nabla)$ has unipotent monodromy by Lemma~\ref{lem:reduction}. Let $\sigma$ be a section of $H^0((\D^{\ast})^n, F^q\cV)$ such that 
\[ \int_{(\D^{\ast})^n} |\sigma|^2_h  d\mu < +\infty.\]
By Hartog's theorem, to show $\sigma \in H^0(\D^n,\cV^{>-1})$, it suffices to show that $\sigma$ extends over any generic point of each divisor $D_j$. 

Without loss of generality, we only need to prove this for $D_1$.  Around a generic point of $D_1$, we can choose a coordinate neighborhood $U_1$ such that it is isomorphic to $\D^{\ast} \times \D^{n-1}$. Let $s_1$ be the coordinate of $\D^{\ast}$. Let $T$ be the monodromy operator of $\cV|_{U_1}$ and $N \colonequals \log T$.  Let $p: \mathbb{H}\times \D^{n-1} \to \D^{\ast} \times \D^{n-1}$ be the universal covering map and let $\Phi: \mathbb{H}\times \D^{n-1} \to \mathbf{D}$ be the corresponding period map. Let $V=H^0(\mathbb{H}\times \D^{n-1}, p^{\ast}\cV)^{p^{\ast}\nabla}$ and let $W_{\bullet}V$ be the weight filtration associated with $N$. 

Let $\D^{\ast}_r \colonequals \{ s_1\in \D^{\ast} \colon \abs{s_1}<r\}$ be the punctured disk of radius $r<1$. Choose $\{\sigma_\alpha\}$ to be a frame section of $F^q\cV^{>-1}$ as in Lemma \ref{lemma:refined version of Schmids Theorem}. Write
\[ \sigma\big|_{U_1} = \sum_\alpha h_\alpha \sigma_\alpha,\]
where $h_\alpha$ is a holomorphic function over $\D^{\ast}_r \times \D^{n-1}$. It suffices to prove that $h_\alpha$ extends to a holomorphic function over $\D_r \times \D^{n-1}$.

Since $\sigma$ is a section of the smallest nonzero piece of $F^q\cV^{>-1}$ over $\Delta_r^{\ast}\times \Delta^{n-1}$, Theorem \ref{thm:one dimension hodge metric} and Lemma \ref{lemma:refined version of Schmids Theorem} imply that there is a positive constant $C_1$ such that over $\D^{\ast}_r \times \D^{n-1}$
\[ \abs{\sigma}^2_{h} \geq C_1\sum_{n_\alpha\geq 0} |h_\alpha|^2(-\log \abs{s_1})^{n_{\alpha}}. \]

Therefore
\begin{align*}
 \infty &> \int_{\D^{\ast}_r \times \D^{n-1}} \abs{\sigma}^2_h d\mu\\
          &>C_1\int_{\D^{\ast}_r \times \D^{n-1}} \sum_{n_{\alpha} \geq 0} |h_\alpha|^2(-\log \abs{s_1})^{n_{\alpha}}  d\mu.
\end{align*}
We can rescale $r$ to $1$, then by Lemma \ref{lem:L2holomorphic}, $h_\alpha$ is holomorphic over $\D_r\times \D^{n-1}$. Therefore $\sigma$ extends to a holomorphic section of $\cV^{>-1}$ over $\D^{n}$.

\end{proof}

To conclude the proof, we need the following
\begin{lemma}\label{lem:L2holomorphic}
Let $f$ be a holomorphic function on $\D^{\ast}\times \D^{n-1}$. Let $s_1$ be the coordinate of $\D^{\ast}$. Suppose there exists an integer $k\geq 0$ such that 
\[ \int_{\D^{\ast}\times \D^{n-1}} (-\log \abs{s_1})^k\cdot \abs{f}^2 d\mu <\infty, \]
Then $f$ must be a holomorphic function on $\D \times \D^{n-1}$.
\end{lemma}

\begin{proof}
We denote $d\mu_n=c_n ds_1\wedge d\bar{s}_1\wedge \cdots \wedge ds_n\wedge d\bar{s}_n$ to be the standard Lebesgue measure on $\D^n$ and $c_n=2^{-n}(-1)^{n^2/2}$. Then $d\mu_n=d\mu_1\times d\mu_{n-1}$.
Expand $f$ as a Laurent series in $s_1$:
\[ f = \sum_{i\in \mathbb{Z}} f_is_1^i,\]
where $f_i$ is a holomorphic function on $\D^{n-1}$.

To simplify the presentation, we denote $-\log \abs{s_1}$ by $L(s_1)$. Let $(r,\theta)$ be the polar coordinate such that $s_1=re^{\sqrt{-1}\theta}$. Observe that $L(s_1)$ is a function only depending on $\abs{s_1}=r$, then for any integers $i,j$ such that $i\neq j$,
\[ \int_{\D^{\ast}} L(s_1)^k s_1^i\bar{s}_1^j d\mu_1 = \int_{0}^1 L(r)^kr^{i+j+1} dr \int_{0}^{2\pi} e^{\sqrt{-1}\theta(i-j)}d\theta =0. \] 
Since the Laurent series converges on any anulus $\{ \epsilon_1 \leq \abs{s_1} \leq \epsilon_2 \} \times \D^{n-1}$, so
\begin{eqnarray*}
\int_{\D^{\ast}\times \D^{n-1}} L(s_1)^k\cdot \abs{f}^2 d\mu &=& \int_{\D^{n-1}}\left(\int_{\D^{\ast}} L(s_1)^k \cdot \sum_{i,j} s_1^i\bar{s}_1^j d\mu_1\right)f_i\bar{f}_j d\mu_{n-1} \\
&=& \sum_{i\in \mathbb{Z}} \left(\int_{\D^{\ast}} L(s_1)^k \abs{s_1}^{2i}d\mu_1\right)\cdot \left( \int_{\D^{n-1}}\abs{f_i}^2 d\mu_{n-1}\right) \\
\end{eqnarray*}
Since $\int_{\D^{\ast}} L(s_1)^k\abs{s_1}^{2i} d\mu < \infty$ if and only if $i\geq 0$ and $\int_{\D^{n-1}}\abs{f_i}^2 d\mu_{n-1}=0$ if and only if $f_i\equiv 0$, we must have $f_i \equiv 0$ for all  $i <0$.
In particular we conclude that $f$ is holomorphic over $\D\times \D^{n-1}$.

\end{proof}

\subsection{Proof of Theorem \ref{thm:main}: second half} 
In this section, we prove the second half of Theorem \ref{thm:main}. We work with Set-up \ref{setup for the proof}. Let $j:X\setminus D\to X$ to be the open embedding. Denote $\mathcal{F}$ to be the subsheaf of $j_{\ast}(F^q\cV)$ consisting of sections $\sigma$ of $F^q\cV$ which are locally $L^2$ near $D$. This means that for any point in $D$, there exists a coordinate neighborhood $U$ such that $\abs{\sigma}_h^2$ is integrable on $U\setminus D$ with respect to the standard Lebesgue measure on $U$. Recall that $\cM$ is the Hodge module on $X$ associated to $\cV$. In \S \ref{sec:first half}, we have proved that $\mathcal{F} \subseteq F_{-q}\cM$. It suffices to show 
\begin{eqnarray}\label{eqn:supset}
F_{-q}\cM\subseteq \mathcal{F} .
\end{eqnarray}
Since (\ref{eqn:supset}) is a local statement, we can assume $X=\D^n$. Let $\alpha \in H^0(X,F_{-q}\cM)$. Since $F_{-q}\cM|_{U}=F^{q}\cV$, we can restrict the section $\alpha$ to a section $\sigma \in H^0(X\setminus D, F^q\cV)$. Moreover, $\sigma$ uniquely determines $\alpha$ because $F_{-q}\cM$ is torsion-free.  We will show that $\abs{\sigma}^2_h$ is integrable on $X$ with respect to the Lebesgue measure.
 
\textbf{Step 1}. If $D$ is a simple normal crossing divisor, choose $s_1,\ldots,s_n$ to be the coordinates on $X$ such that $D=\cup \{ s_i=0 \}$. By the asymptotics of Hodge metrics proved by Cattani-Kaplan-Schmid \cite{CKS} (a precise formulation can be found at \cite[Theorem 5.1]{CK82}), over any region of the form
\[ \{\underline{s}=(s_1,\ldots,s_n) \in (\D^{\ast})^n : |\underline{s}|< a <1; 
\frac{\log|s_j|}{\log|s_{j+1}|} \geq \epsilon, 1\leq j \leq n-1 \},\]
$\abs{\sigma}^2_h$ is bounded above by sums of products of logarithm functions in $\abs{s_i}$. In particular, it is integrable with respect to the standard Lebesgue measure on this region. Since we can cover the neighborhood of any point of $D$ using finite regions of this type, it follows that $\abs{\sigma}^2_h$ is integrable on $X$ with respect to the Lebesgue measure.

\textbf{Step 2}. If $D$ is an arbitrary divisor, choose a log resolution 
\[ f: (\tilde{X},\tilde{D}) \to (X,D),\]
such that $\tilde{D}=f^{\ast}(D)$ is simple normal crossing and $\tilde{X}\setminus \tilde{D} \cong  X\setminus D$. The VHS $\cV$ pulls back isomorphically to a VHS $\tilde{\cV}$ on $\tilde{X}\setminus \tilde{D}$. Abusing notations, we use $h$ to denote the Hodge metric on $\tilde{\cV}$. Let $\tilde{\cM}$ be the Hodge module on $\tilde{X}$ uniquely determined by $\tilde{\cV}$ via Saito's Structure Theorem \ref{thm:structure}. Let $s_1,\ldots,s_n$ be the standard coordinates on $X=\Delta^n$ and set $\xi\colonequals \alpha\otimes ds_1\wedge \cdots \wedge ds_n$ to be the section in $H^0(X,F_{-q}\cM\otimes \omega_X)=H^0(\tilde{X},F_{-q}\tilde{\cM}\otimes \omega_{\tilde{X}})$, where the equality follows from the direct image Theorem \eqref{eqn:direct image}. Over $X\setminus D$, since the restriction of $\alpha$ is $\sigma$, we define $\xi\wedge \overline{\xi}\colonequals \abs{\sigma}_h^2 d\mu$, where $d\mu$ is the standard Lebesgue measure. Our goal is to show that the integral
\[ \int_{X\setminus D}\abs{\sigma}^2_hd\mu=\int_{X\setminus D}\xi\wedge \overline{\xi}\]
is finite. Since $\xi$ can be viewed as a section of $H^0(\tilde{X},F_p\tilde{\cM}\otimes \omega_{\tilde{X}})$, locally on a standard coordinate neighborhood $\Delta^n$ on $\tilde{X}$ with coordinates $\tilde{s}_1,\ldots,\tilde{s}_n$, we can write $\xi=\tilde{\alpha}\otimes d\tilde{s}_1\wedge \cdots \wedge d\tilde{s}_n$, such that $\tilde{\alpha}\in H^0(\tilde{X},F_{p}\tilde{\cM})$. Then the restriction of $\tilde{\alpha}$ to $\tilde{X}\setminus \tilde{D}$ is a section $\tilde{\sigma}\in H^0(\tilde{X},F^q\tilde{\cV})$. Since $f$ induces an isomorphism $\tilde{X}\setminus \tilde{D}\cong X\setminus D$, we have
\begin{equation}\label{eqn: integral on X tilde}
  \int_{X\setminus D}\xi\wedge \overline{\xi}=\int_{\tilde{X}\setminus \tilde{D}}\abs{\tilde{\sigma}}_h^2f^{\ast}(d\mu).
  \end{equation}
Because the log resolution map $f$ is a polynomial map, the measure $f^{\ast}(d\mu)$ is locally bounded by a multiple of the standard Lebesgue measure $d\tilde{\mu}$ on $\tilde{X}$. Using Step 1, we know that 
\[ \int_{\Delta^n\setminus \tilde{D}}\abs{\tilde{\sigma}}^2_hd\tilde{\mu} <+\infty\]
Since $\tilde{X}$ can be covered by finitely many such standard coordinate polydisks, it follows that the integral \eqref{eqn: integral on X tilde} is finite. Therefore, we conclude the proof of Theorem \ref{thm:main}.

\bibliographystyle{amsalpha}
\bibliography{mep}{}
\vspace{2cm}

\footnotesize{
\textsc{Department of Mathematics, Stony Brook University, Stony Brook, New York 11794} \\
\indent \textit{E-mail address:} \href{mailto:cschnell@math.stonybrook.edu}{cschnell@math.stonybrook.edu}

\vspace{\baselineskip}

\textsc{Department of Mathematics, Stony Brook University, Stony Brook, New York 11794} \\
\indent \textit{E-mail address:} \href{mailto:ruijie.yang@stonybrook.edu}{ruijie.yang@stonybrook.edu}

\end{document}